\documentclass[oneside,english]{amsart}
\usepackage[T1]{fontenc}
\usepackage[latin9]{inputenc}
\usepackage{amsthm}
\usepackage{amstext}
\usepackage{amssymb}

\makeatletter
\numberwithin{equation}{section}
\numberwithin{figure}{section}
\theoremstyle{plain}
\newtheorem{thm}{\protect\theoremname}
  \theoremstyle{plain}
  \newtheorem{lem}[thm]{\protect\lemmaname}
  \theoremstyle{plain}
  \newtheorem{cor}[thm]{\protect\corollaryname}
  \theoremstyle{plain}
  \newtheorem{prop}[thm]{\protect\propositionname}

\makeatother

\usepackage{babel}
  \providecommand{\corollaryname}{Corollary}
  \providecommand{\lemmaname}{Lemma}
  \providecommand{\propositionname}{Proposition}
\providecommand{\theoremname}{Theorem}

\begin{document}
\global\long\def\F{\mathcal{F}}
\global\long\def\G{\mathcal{G}}
\global\long\def\H{\mathcal{H}}
\global\long\def\U{\mathcal{U}}
\global\long\def\S{\mathcal{S}}
\global\long\def\filter{\mathbb{F}}
\global\long\def\lm{\mathrm{lim}}
\global\long\def\adh{\mathrm{adh}}
\global\long\def\N{\mathcal{N}}
\global\long\def\A{\mathcal{A}}
\global\long\def\then{\Longrightarrow}

\title[regularity in CAP]{Function spaces and contractive extensions in approach theory: the role of regularity}

\author{Eva Colebunders, Fr\'ed\'eric Mynard and William Trott}
\begin{abstract}
Two classical results characterizing regularity of a convergence space
in terms of continuous extensions of maps on one hand, and in terms
of continuity of limits for the continuous convergence on the other,
are extended to convergence-approach spaces. Characterizations are
obtained for two alternative extensions of regularity to convergence-approach
spaces: regularity and strong regularity. The results improve upon
what is known even in the convergence case. On the way, a new notion
of strictness for convergence-approach spaces is introduced. 
\end{abstract}


\address{Vakgroep Wiskunde, Vrije Universiteit Brussel, Pleinlaan 2, 1050 Brussel , Belgi\"{e}
 Mathematical Sciences, POBOX 8093, Georgia Southern University, Statesboro,
GA 30460, USA.}

\email{evacoleb@vub.ac.be; fmynard@georgiasouthern.edu; wtrott@georgiasouthern.edu}

\maketitle

\section{Preliminaries}

For a set $X$, we denote the set of all filters on $X$ as $\filter X$,
and the set of all ultrafilters on $X$ as $\mathbb{U}X$. We order
$\mathbb{F}X$ by inclusion of families. We denote the set of ultrafilters
finer than some $\F\in\filter X$ as $\mathbb{U}(\F)$, and the principal
filter of $A\subseteq X$ is denoted $\{A\}^{\uparrow}$. More generally,
if $\A\subseteq2^{X}$ then $\A^{\uparrow}:=\{B\subseteq X:\exists A\in\A,\, A\subseteq B\}$.
If $f:X\rightarrow Y$ and $\F\in\filter X$, then the image filter
is denoted $f[\F]:=\{f(F):F\in\F\}^{\uparrow}$. Note that if $\F\in\mathbb{F}X$
and $\S:X\to\mathbb{F}Y$ then 
\[
\S(\F):=\bigcup_{F\in\F}\bigcap_{x\in F}\S(x)
\]
is a filter on $Y$.

For some relation $\xi$ between $\filter X$ and $X$, for $x\in X$
and $\F\in\filter X$, we say that $\F$ converges to $x$ in $(X,\xi)$,
denoted by $x\in\lm_{X}\F$, or $x\in\lm\F$ if no ambiguity arises
(%
\footnote{More generally, we can extend convergence to filter-bases by $\lim\mathcal{B}:=\lim\mathcal{B}^{\uparrow}$,
so that we do not distinguish between a filter-base and the filter
it generates unless explicitly needed.%
}), if $(\F,x)\in\xi$. We say $(X,\xi)$ is a \emph{preconvergence
space} if it satisfies the following conditions: 
\[
\F\leq\G\Longrightarrow\lm\F\subseteq\lm\G;
\]
\[
\forall\F,\,\G\in\filter X,\,\lm(\F\wedge\G)\supseteq\lm\F\cap\lm\G,
\]
and a \emph{convergence }if it additionally satisfies
\[
\forall x\in X,\, x\in\lm\{x\}^{\uparrow}.
\]
Of course, every topology can be considered as a convergence, setting
$x\in\lim\F$ if and only if $\F\supseteq\N(x)$, where $\N(x)$ denotes
the neighborhood filter of $x$.

Given convergence spaces $X$ and $Y$, a map $f:X\rightarrow Y$
is \emph{continuous} if 
\[
f(\lm_{X}\F)\subseteq\lm_{Y}f[\F].
\]
The category with convergence spaces as objects and continuous maps
as morphisms is denoted \textbf{Conv}. 

Convergence-approach spaces, introduced in \cite{lowen1}, generalize
convergence spaces.
The pair $(X,\lambda)$ is called a \emph{preconvergence-approach
space} if $\lambda:\filter X\rightarrow[0,\infty]^{X}$ (where $[0,\infty]^{X}$ is ordered pointwise) satisfies
(%
\footnote{This condition implies in particular that 
\[
\F\leq\G\implies\lambda(\F)\geq\lambda(\G).
\]
})

\[
\forall\F,\G\in\filter X,\,\lambda(\F\wedge\G)=\lambda(\F)\vee\lambda(\G),
\]
 and a \emph{convergence-approach space}, or shortly a CAP space, if it additionally
satisfies
\begin{equation}
\forall x\in X,\,\lambda(\{x\}^{\uparrow})(x)=0.\label{eq:centered}
\end{equation}
We often denote a CAP space $(X,\lambda)$ simply by $X$ and we denote
by $\lambda_{X}$ the associated limit. 

A function $f:X\rightarrow Y$ between two convergence-approach spaces
is a \emph{contraction} if 
\[
\lambda_{Y}(f[\F])\circ f\leq\lambda_{X}(\F)
\]
 for every $\F\in\filter X$. We denote by \textbf{Cap }the category
of convergence-approach spaces and contractions. This is a topological
category over \textbf{Set}, so that subspaces and products are defined
as initial lifts of their corresponding structured sources in \textbf{Set}.

Each convergence space $(X,\xi)$ can also be considered a convergence-approach
space by letting 
\[
\lambda_{X}(\F)(x)=\begin{cases}
0 & \mbox{if \ensuremath{x\in\lm_{X}\F}}\\
\infty & \mbox{otherwise.}
\end{cases}
\]
A map $f:X\to Y$ between two convergence spaces is continuous if
and only if it is a contraction when $X$ and $Y$ are considered
as convergence-approach spaces.

\textbf{Conv} is embedded as a concretely reflective and concretely
coreflective subcategory of \textbf{Cap}: If $X$ is a convergence-approach
space, the \textbf{Conv}-coreflection $c(X)$ is defined by $x\in\lm_{c(X)}\F$
if and only if $\lambda(\F)(x)=0$, while the \textbf{Conv}-reflection
$r(X)$ is defined by $x\in\lm_{r(X)}\F$ if and only if $\lambda(\F)(x)<\infty.$

Given a CAP space $X$, a subset $A$ of $X$ and $\epsilon\geq0$,
let 
\[
A^{(\epsilon)}=\{x\in X\,:\,\exists\U\in\mathbb{U}X,\, A\in\U,\,\lambda(\U)(x)\leq\epsilon\},
\]
and $\A^{(\epsilon)}:=\{A^{(\epsilon)}\,:\, A\in\A\}$ whenever $\A\subseteq2^{X}$.
Note that 
\begin{eqnarray*}
A^{(0)} & = & \adh_{c(X)}A:=\bigcup_{\U\in\mathbb{U}(A)}\lm_{c(X)}\U\\
\A^{(0)} & = & \adh_{c(X)}^{\natural}\A:=\{\adh_{c(X)}A:\, A\in\A\}.
\end{eqnarray*}

Let $\oplus:[0,\infty]\times[0,\infty]\to[0,\infty]$ be a commutative
and associative binary operation such that 
\begin{equation}
0\oplus r=r\label{eq:neutral}
\end{equation}
 
\begin{equation}
r\oplus\bigwedge_{a\in A}a=\bigwedge_{a\in A}\left(r\oplus a\right)\label{eq:frameforgeq}
\end{equation}
 for every $r\in[0,\infty]$, and $A\subseteq[0,\infty]$ (%
\footnote{that is, $[0,\infty]$ is a quantale for the reverse order, with neutral
element $0$.%
}). The two leading examples are the usual addition $+$ of $[0,\infty]$
(where $a+\infty=\infty$ for all $a\in[0,\infty]$), and pairwise
maximum $\vee$. 
Note that (\ref{eq:frameforgeq}) ensures that $\oplus$ preserves
order, that is, 
\begin{equation}
a\leq b\text{ and }c\leq d\then a\oplus c\leq b\oplus d,\label{eq:preservesorder}
\end{equation}
and thus $\vee\leq\oplus$, by combining (\ref{eq:preservesorder})
with (\ref{eq:neutral}). Moreover, (\ref{eq:frameforgeq}) also gives
\begin{equation}
(a+\epsilon)\oplus(b+\epsilon)\underset{\epsilon\to0}{\to}a\oplus b.\label{eq:limit}
\end{equation}

A convergence-approach space $(X,\lambda)$ is \emph{$\oplus$-regular}
if 
\begin{equation}
\lambda(\F^{(\epsilon)})\leq\lambda(\F)\oplus\epsilon\label{eq:CAPregepsilon}
\end{equation}
for all $\F\in\filter X$ and $\epsilon\in[0,\infty]$. $+$-regularity
and $\vee$-regularity are usually called \emph{regularity} and \emph{strong
regularity }respectively \cite{BrockKent}. An alternative characterization
of $\oplus$-regularity is available (see, e.g., \cite{BrockKent}):
a CAP space $(X,\lambda)$ is $\oplus$-regular if and only if 
\begin{equation}
\lambda(l[\F])\leq\lambda(\S(\F))\oplus\bigvee_{a\in A}\lambda(\S(a))(l(a))\label{eq:CAPregselection}
\end{equation}
for every $A\neq\emptyset$, $l:A\rightarrow X$, $\F\in\filter A$,
and $\S:A\rightarrow\filter X$. In particular, a convergence space
is \emph{regular }if it is regular (equivalently strongly regular)
when considered as a CAP space (%
\footnote{that is, if $\lim\F\subseteq\lim(\adh^{\natural}\F)$ for every filter
$\F$, equivalently, if $\lim\S(\F)\subseteq\lim l[\F]$ for every
$A\neq\emptyset$, $l:A\rightarrow X$, $\F\in\filter A$, and $\S:A\rightarrow\filter X$
with $l(a)\in\lim\S(a)$ for each $a\in A$.%
}).

Regularity can be localized: we call a point $x$ of a CAP space a
\emph{$\oplus$-regularity point }if (\ref{eq:CAPregepsilon}), equivalently
(\ref{eq:CAPregselection}), is satisfied at $x$.

A CAP space $(X,\lambda)$ is \emph{$\oplus$-diagonal} if 
\[
\lambda(\S(\F))\leq\lambda(\F)\oplus\bigvee_{a\in X}\lambda(\S(a))(a)
\]
for every $\F\in\filter X$ and $\S:X\rightarrow\filter X$. Note
that a $+$-diagonal CAP space is usually called \emph{diagonal.}

A \emph{pre-approach space }$X$ is a CAP space satisfying 
\[
\lambda\left(\bigcap_{\F\in\mathbb{D}}\F\right)=\bigvee_{\F\in\mathbb{D}}\lambda\F,
\]
for every $\mathbb{D}\subseteq\mathbb{F}X$. A convergence space that
is a pre-approach space when considered a CAP space is called \emph{pretopological.}
A pre-approach space that is $+$-diagonal is called an \emph{approach
space }(e.g., \emph{\cite{lowen1}})\emph{. }More generally, we call
$\oplus$-\emph{approach space} a pre-approach space that is $\oplus$-diagonal.
$\vee$-approach spaces are also called \emph{non-archimedean approach
spaces}, because a metric approach space is non-archimedan in the
metric sense if and only if it is non-archimedean in the \textbf{Cap
}sense \cite{BrockKent}. When a (pre)topological space is considered
as a (pre)approach space, or a convergence-approach space, we may
talk of a (pre)topological (pre)approach space, or a (pre)topological
CAP space.

\section{Regularity and Hom-structures}

Given CAP spaces $(X,\lambda_{X})$ and $(Y,\lambda_{Y})$, consider
on the set $C(X,Y)$ of contractions from $X$ to $Y$ the approach
limit 
\[
\lambda_{[X,Y]}(\F)(f):=\inf\{\alpha\in[0,\infty]:\,\forall\G\in\mathbb{F}X,\,\lambda_{Y}\left\langle \G,\F\right\rangle \circ f\leq\lambda_{X}(\G)\vee\alpha\}
\]
where $\langle\G,\F\rangle:=\{\langle G,F\rangle:G\in\G,\, F\in\F\}^{\uparrow}$
and $\langle G,F\rangle:=\{h(x):h\in F,\, x\in G\}$ (%
\footnote{We could consider a variant with a tensor $\oplus$ on $[0,\infty]$
verifying (\ref{eq:neutral}) and (\ref{eq:frameforgeq}):
\[
\lambda_{[X,Y]_{\oplus}}(\F)(f):=\inf\{\alpha\in[0,\infty]:\,\forall\G\in\mathbb{F}X,\,\lambda_{Y}\left\langle \G,\F\right\rangle (f(\cdot))\leq\lambda_{X}(\G)(\cdot)\oplus\alpha\},
\]
but that would make little difference in our results, even though
it would complicate statements by forcing the introduction of two
different tensors on $[0,\infty]$. Thus we prefer focusing on the
canonical hom-structure $[X,Y]$ of \textbf{Cap}. %
}). This defines a CAP space $[X,Y]$ with underlying set $C(X,Y)$
that turns \textbf{Cap }into a cartesian-closed category. Note that
$\lambda_{[X,Y]}$ makes sense when defined on $Y^{X}$ even though
the resulting structure is a preconvergence approach that satisfies
(\ref{eq:centered}) only at functions $f\in C(X,Y)$. We will not make a notational
distinction between $\lambda_{[X,Y]}$ defined on $C(X,Y)$ or on
$Y^{X}$. If $X$ and $Y$ are convergence spaces considered as CAP
spaces, then $r[X,Y]=c[X,Y]$ is the usual continuous convergence,
that we will also denote $[X,Y]$.

We define the \emph{$\oplus$-default of contraction} of a function
$f\in Y^{X}$, denoted $m_{\oplus}(f)$, in the following way:
\[
m_{\oplus}(f):=\inf\{\alpha\in[0,\infty]\,:\,\forall\G\in\filter X,\,\lambda_{Y}(f[\G])\circ f\leq\lambda_{X}(\G)\oplus\alpha\}.
\]
 It is clear that a function $f$ is a contraction if and only if
$m_{\oplus}(f)=0$. Note that 
\[
m_{+}(f)\leq m_{\vee}(f).
\]

\begin{thm}
\label{thm:continuouslimitsCAP}If $Y$ is a $\oplus$-regular convergence-approach
space, $X$ is a convergence-approach space, and $f\in Y^{X}$ then
\[
m_{\oplus}(f)\leq\left(\bigwedge_{\F\in\filter(Y^{X})}\lambda_{[X,Y]}(\F)(f)\right)\oplus\left(\bigwedge_{\F\in\filter(Y^{X})}\lambda_{[X,Y]}(\F)(f)\right).
\]
Conversely, if $Y$ is not $\oplus$-regular, there is a \emph{topological}
space $X$ and $f\in Y^{X}$ with 
\[
m_{\oplus}(f)>\bigwedge_{\F\in\filter(Y^{X})}\lambda_{[X,Y]}(\F)(f).
\]
 
\end{thm}
We will need the following observation to prove Theorem \ref{thm:continuouslimitsCAP}.
\begin{lem}
\label{lem:reg}If $\alpha\in[0,\infty]$, $\G\in\filter X$, $\F\in\filter(Y^{X})$
and $f\in Y^{X}$ satisfy 
\[
\lambda_{Y}(\langle \{x\}^{\uparrow},\F\rangle)(f(x))\leq\alpha,
\]
 for every $x\in X$, then 
\[
f[\G]\geq\langle\G,\F\rangle^{(\alpha)}.
\]

\end{lem}
Note that the case $\alpha=0$ states that if $\F$ converges pointwise
to $f\in Y^{X}$ then for any $\G\in\mathbb{F}X$, $f[\G]\geq\adh_{c(Y)}^{\natural}\langle\G,\F\rangle$.
\begin{proof}
Let $x\in G$ for some $G\in\G$. We consider the filter $\langle\{x\}^{\uparrow},\F\rangle$
on $\langle G,F\rangle$ for $F\in\F$. Then by the assumption, 
\[
\lambda_{Y}(\langle\{x\}^{\uparrow},\F\rangle)(f(x))\leq\alpha,
\]
 so $f(x)\in\langle G,F\rangle^{(\alpha)}$. Thus $f(G)\subseteq\langle G,F\rangle^{(\alpha)}$
for any $G\in\G$ and $F\in\F$, so $f[\G]\geq\langle\G,\F\rangle^{(\alpha)}$.
\end{proof}

\begin{proof}[Proof of Theorem \ref{thm:continuouslimitsCAP}]
Let $Y$ be a $\oplus$-regular convergence-approach space and let
\[
c:=\bigwedge_{\F\in\filter(Y^{X})}\lambda_{[X,Y]}(\F)(f).
\]
 For $\epsilon>0$, there is an $\F_{\epsilon}\in\filter(Y^{X})$
such that $\lambda_{[X,Y]}(\F_{\epsilon})(f)<c+\epsilon$, and, by
definition of $\lambda_{[X,Y]}$, there is $\alpha_{\epsilon}<\lambda_{[X,Y]}(\F_{\epsilon})(f)+\epsilon<c+2\epsilon$
such that $\lambda\left\langle \G,\F_{\epsilon}\right\rangle \circ f\leq\lambda_{X}(\G)\vee\alpha_{\epsilon}$
for every $\G\in\mathbb{F}X$. In particular,
\[
\lambda\left\langle \{x\}^{\uparrow},\F_{\epsilon}\right\rangle (f(x))\leq\lambda_{X}(\{x\}^{\uparrow})(x)\vee\alpha_{\epsilon}=\alpha_{\epsilon}
\]
so that $f[\G]\geq\langle\G,\F_{\epsilon}\rangle^{(\alpha_{\epsilon})}$
by Lemma \ref{lem:reg}.
Since $Y$ is $\oplus$-regular, 
\begin{eqnarray*}
\lambda_{Y}(\langle\G,\F_{\epsilon}\rangle^{(\alpha_{\epsilon})})\circ f\leq\lambda\left\langle \G,\F_{\epsilon}\right\rangle \circ f\oplus\alpha_{\epsilon} & \leq & \left(\lambda_{X}(\G)\vee\alpha_{\epsilon}\right)\oplus\alpha_{\epsilon},\\
 & \leq & \lambda_{X}(\G)\oplus(\alpha_{\epsilon}\oplus\alpha_{\epsilon}),
\end{eqnarray*}
using (\ref{eq:preservesorder}) and the fact that $\vee\leq\oplus$.
Thus $\lambda_{Y}(f[\G])\circ f\leq\lambda_{X}(\G)\oplus(\alpha_{\epsilon}\oplus\alpha_{\epsilon})$
because $f[\G]\geq\langle\G,\F_{\epsilon}\rangle^{(\alpha_{\epsilon})}$.
However, $\alpha_{\epsilon}\oplus\alpha_{\epsilon}<(c+2\epsilon)\oplus(c+2\epsilon)$,
and since $\epsilon$ is arbitrary, the inequality becomes $\lambda_{Y}(f[\G])\circ f\leq\lambda_{X}(\G)\oplus(c\oplus c)$
by (\ref{eq:limit}), and we conclude that $m_{\oplus}(f)\leq c\oplus c$.

For the converse, assume that $Y$ is not $\oplus$-regular. Then
there exists $A\neq\emptyset$, $l:A\rightarrow Y$, $\S:A\rightarrow\filter Y$,
$\H\in\filter A$, and $y_{0}\in Y$ such that 
\begin{equation}
\lambda_{Y}(l[\H])(y_{0})>\lambda_{Y}(\S(\H))(y_{0})\oplus\bigvee_{a\in A}\lambda_{Y}(\S(a))(l(a)).\label{eq:prob}
\end{equation}

Let $X:=(Y\times A)\cup A\cup\{x_{\infty}\}$ where $x_{\infty}\notin A$.
Define $p_{Y}:Y\times A\to Y$ by $p_{Y}(y,a)=y$ for all $(y,a)\in Y\times A$,
and let $f:X\rightarrow Y$ be defined by $f_{|A}=l$, $f_{|Y\times A}=p_{Y}$,
and $f(x_{\infty})=y_{0}$. Let 
\[
\N:=\bigcup\limits _{H\in\H}\bigcap\limits _{a\in H}(\S(a)\times\{a\}^{\uparrow})\wedge\{a\}^{\uparrow}.
\]
Now we define $\lambda_{X}$ by, for all $a\in A$ and $y\in Y$:
\begin{eqnarray*}
\lambda_{X}(\G)((y,a)) & := & \begin{cases}
0 & \mbox{if }\G=\{(y,a)\}^{\uparrow}\\
\infty & \mbox{otherwise}
\end{cases}\\
\lambda_{X}(\G)(a) & := & \begin{cases}
0 & \mbox{if \ensuremath{\G\geq\left(\S(a)\times\{a\}^{\uparrow}\right)\wedge\{a\}^{\uparrow}}}\\
\infty & \mbox{otherwise}
\end{cases}\\
\lambda_{X}(\G)(x_{\infty}) & := & \begin{cases}
0 & \mbox{if }\G\geq\N\wedge\{x_{\infty}\}^{\uparrow}\\
\infty & \mbox{otherwise}.
\end{cases}
\end{eqnarray*}
Note that $X$ is then a topological CAP space 
and that $$m_{\oplus}(f)>\lambda_{Y}(\S(\H))(y_{0})\oplus\bigvee_{a\in A}\lambda_{Y}(\S(a))(l(a)),$$
because of (\ref{eq:prob}) and $f[\H]=l[\H]$.\\
Let $P:=\{h\in Y^{X}:\, h_{|Y\times A}=p_{Y}\text{ and \ensuremath{h(x_{\infty})=y_{0}}}\}$,
and for each $a\in A$, let 
\begin{eqnarray*}
\hat{a}:Y^{X} & \to & Y\\
h & \mapsto & h(a).
\end{eqnarray*}
Let 
\[
\A:= \bigcup_{a \in A}\left\{ \hat{a}^{-1}(S)\cap P:\,\, S\in\S(a)\right\} 
\]
\[
\mathcal{B}:=\bigcup_{H\in \H}\{\bigcap_{a\in H}\hat{a}^{-1}(S_{a}^{H})\cap P:\,\, (S_{a}^{H})_{a \in H}\in\prod_{a\in H}\S(a)\}.
\]
Then $\A\cup\mathcal{B}$ has the finite intersection property, and
thus generates a filter $\F_{0}$ on $Y^{X}$. It suffices to show
that 
\begin{eqnarray*}
\lambda_{[X,Y]}(\F_{0})(f) & \leq & \lambda_{Y}(\S(\H))(y_{0})\vee\bigvee_{a\in A}\lambda_{Y}(\S(a))(l(a))
\end{eqnarray*}
To this end, by definition of $[X,Y]$, we only need to show that
\[
\lambda_{Y}(\langle\G,\F_{0}\rangle)(f(x))\leq\lambda_{Y}(\S(\H))(y_0)\vee\bigvee_{a\in A}\lambda_{Y}(\S(a))(l(a))\vee\lambda_{X}(\G)(x)
\]
 for every $\G\in\mathbb{F}X$ and $x\in X$.

If $\G\geq\N\wedge\{x_{\infty}\}^{\uparrow}$ then $\langle\G,\F_{0}\rangle\geq\S(\H)\wedge\{y_{0}\}^{\uparrow}$:
for each $B\in\S(\H)$ there is $H \in\H$ and for each $a\in H$,
there is $S_{a}^{H}\in\S(a)$ such that $\bigcup_{a\in H}S_{a}^{H}\subseteq B$
and 
\[
\left\langle \bigcup_{a\in H}((S_{a}^{H}\times\{a\})\cup\{a\}),\bigcap_{a\in H}\hat{a}^{-1}(S_{a}^{H})\cap P\right\rangle \subseteq\bigcup_{a\in H}S_{a}^{H}.
\]
Thus 
\[
\lambda_{Y}(\langle\G,\F_{0}\rangle)(y_{0})\leq\lambda_{Y}(\S(\H))(y_{0})=\lambda_{Y}(\S(\H))(y_{0})\vee\lambda_{X}(\G)(x_{\infty}).
\]

If $\G\geq\left(\S(a)\times\{a\}^{\uparrow}\right)\wedge\{a\}^{\uparrow}$
for some $a\in A$, then 
\[
\langle\G,\F_{0}\rangle\geq\S(a).
\]
Indeed,  by definition of $P$, $\left\langle S\times\{a\} \cup \{a\},\hat{a}^{-1}(S)\cap P\right\rangle \subseteq S$
for any $S\in\S(a)$. 
Thus, taking into account that $f(a)=l(a)$ 
\[
\lambda_{Y}(\langle\G,\F_{0}\rangle)(f(a))\leq\lambda_{Y}(\S(a))(l(a))=\lambda_{Y}(\S(a))(l(a))\vee\lambda_{X}(\G)(a).
\]

 Finally, if $\G$ is a principal ultrafilter $\{t\}^{\uparrow}$ with $t = (y,a)$
then $\langle\{t\}^{\uparrow},\F_{0}\rangle=\{f(t)\}^{\uparrow} = y^{\uparrow}$ and its limit value $\lambda_{Y}(\langle\{t\}^{\uparrow},\F_{0}\rangle)(y) = 0$
%

\end{proof}
In particular, considering convergence spaces as convergence-approach
spaces, we get as an immediate corollary:
\begin{cor}
\label{cor:Wolkvar}Let $Y$ be a convergence space. The following
are equivalent:
\begin{enumerate}
\item $Y$ is regular;
\item 
\[
f\in\lm_{[X,Y]}\F\then f\in C(X,Y)
\]
for every convergence space $X$, every $f\in Y^{X}$ and every $\F\in\mathbb{F}(Y^{X})$;
\item 
\[
f\in\lm_{[X,Y]}\F\then f\in C(X,Y)
\]
for every \emph{topological} space $X$, every $f\in Y^{X}$ and every
$\F\in\mathbb{F}(Y^{X})$.
\end{enumerate}
\end{cor}
In particular, this result generalizes \cite[Theorem 2.6]{Wolk} of
Wolk, which establishes the equivalence between (1) and (3), under
the assumption that $Y$ be topological. It also generalizes \cite[Theorem 2.0.7]{evers2}
in the English summary of the thesis \cite{Evers}, which gives the
equivalence between (1) and (3) where $Y$ is assumed to be pretopological,
and $X$ is only pretopological. Moreover, Corollary \ref{cor:Wolkvar}
answers positively the question at the end of \cite{evers2}, whether
the equivalence between (1) and (2) is valid for a general convergence
space $Y$. 
\begin{cor}
If a convergence-approach space $Y$ is regular then for every convergence-approach
space $X$ and $f\in Y^{X}$, 
\[
m_{+}(f)\leq2\bigwedge_{\F\in\filter(Y^{X})}\lambda_{[X,Y]}(\F)(f).
\]
If $Y$ is not regular, there is a \emph{topological} space $X$ and
$f\in Y^{X}$ with
\[
\bigwedge_{\F\in\filter(Y^{X})}\lambda_{[X,Y]}(\F)(f)<m_{+}(f).
\]

\end{cor}

\begin{cor}
A convergence-approach space $Y$ is strongly regular if and only
if for every convergence approach (equivalently, topological) space
$X$ and $f\in Y^{X}$, 
\[
m_{\vee}(f)\leq\bigwedge_{\F\in\filter(Y^{X})}\lambda_{[X,Y]}(\F)(f).
\]

\end{cor}

\section{Regularity and contractive Extensions}

In this section, we investigate the conditions under which a contractive
map $f:S\to Y$, where $S\subseteq X$ and $X$, $Y$ are CAP spaces,
can be extended to a contraction defined on a larger subset of $X$.
Such a result was recently obtained by G. J{\"a}ger \cite{Jaegerextension},
but we can refine his result in a way that improves upon what is known,
even in the case of convergence spaces. As a result, we obtain new
characterizations of regular and of strongly regular CAP spaces in
terms of existence of contractive extensions of maps.

Given two CAP spaces $X$ and $Y$, $x\in X$, $S\subseteq X$, $f:S\to Y$
and $\alpha,\,\epsilon\in[0,\infty]$, define
\begin{eqnarray*}
H_{S}^{\epsilon}(x) & := & \{\F\in\filter S\,:\,\lambda_{X}(\F)(x)\leq\epsilon\}\\
F_{S}^{\epsilon}(x) & := & \{y\in Y\,:\,\forall\F\in H_{S}^{\epsilon}(x),\,\lambda_{Y}(f[\F])(y)\leq\epsilon\}\\
h(S,f,\alpha) & := & \Big\{ x\in S^{(\alpha)}:\,\bigcap_{\epsilon\in[0,\infty]}F_{S}^{\epsilon}(x)\neq\emptyset\Big\}\\
h(S,f) & := & h(S,f,0).
\end{eqnarray*}
Note that $F_{S}^{\epsilon}(x)=Y$ if $H_{S}^{\epsilon}(x)=\emptyset$,
that $S\subseteq h(S,f)\subseteq h(S,f,\alpha)$ for each $\alpha$,
and that if $X$ and $Y$ are convergence spaces (considered as CAP
spaces) then 
\[
h(S,f)=\Big\{ x\in\adh S:\,\bigcap_{\F\in\mathbb{F}S,\, x\in\lim_{X}\F}\lm_{Y}f[\F]\neq\varnothing\Big\}.
\]

Given a contraction $f:S\to Y$ where $S\subseteq X$, and $\alpha\in[0,\infty]$,
we call a function $g:h(S,f,\alpha)\rightarrow Y$ with $g_{|S}=f$
and $g(x)\in\bigcap\limits _{\epsilon\in[0,\infty]}F_{S}^{\epsilon}(x)$
for each $x\in h(S,f,\alpha)$ an \emph{admissible extension} of $f$.
If each $g(x)$ is also a $\oplus$-regularity point, then we call
$g$ a \emph{$\oplus$-regular extension} of $f$. Note that we can
adopt a similar terminology in \textbf{Conv} (%
\footnote{Namely if $f:S\to Y$ is continuous for $S\subset X$, we call a function
$g:h(S,f)\to Y$ with $g_{|S}=f$ and $g(x)\in\underset{x\in\lim_{X}\F;\, S\in\F}{\bigcap}\lim_{Y}f[\F]$
for each $x\in h(S,f)$ an \emph{admissible extension of $f$. }If
moreover each $g(x)$ is a regularity point, $g$ is a \emph{regular
extension of $f$.}%
}).\textbf{ }

Let $X$ be a CAP space and $S\subseteq X$ and $\alpha \in [0,\infty].$ Then $S$ is called an \emph{$\alpha$-$\oplus$-strict}
subspace if for every $x\in S^{(\alpha)}$ and every $\F\in\filter S^{(\alpha)}$ there
is $\G\in\mathbb{F}S$ such that $\G^{(\alpha)}\leq\F$
and 
\begin{equation}
\lambda(\G)(x)\leq\lambda(\F)(x)\oplus \alpha.\label{eq:strict-1}
\end{equation}
$S$ is called \emph{$\oplus$-strict} if it is $\alpha$-$\oplus$-strict for every $\alpha \in [0,\infty]$.

$S$ is called a \emph{uniformly  $\alpha$-$\oplus$-strict}
subspace if for every $\F\in\filter S^{(\alpha)}$ there
is $\G\in\mathbb{F}S$ such that $\G^{(\alpha)}\leq\F$
and
\begin{equation}
\lambda(\G)\leq\lambda(\F)\oplus \alpha.\label{eq:strict-1}
\end{equation}
on $S^{(\alpha)}.$
$S$ is called \emph{uniformly $\oplus$-strict} if it is uniformly $\alpha$-$\oplus$-strict for every $\alpha \in [0,\infty]$.\\

Clearly every subspace $S$ is (uniformly) $\infty$-$\oplus$-strict. 
When the definition of $\alpha$-$\oplus$-strictness is applied to a subspace $S$ of a convergence space $X,$ all cases $\alpha < \infty$ reduce to $\alpha = 0$. In this setting $S$ is $\oplus$-strict if for every $x\in \adh{S}$ and every $\F\in\filter \adh{S}$ there
is $\G\in\mathbb{F}S$ such that $\adh^{\natural}{\G} \leq\F$
and 
\begin{equation}
x \in \textup{lim} \F  \Longrightarrow x \in \textup{lim} \G
\end{equation}
This notion coincides with what is called a strict subspace in the literature on convergence spaces
(%
\footnote{in this setting the definition is usually restricted to dense subsets of $X$, e.g.,
\cite{FricKent92}, but this restriction is inessential.%
}). 

\begin{prop}
\label{prop:diagtostrict}If $X$ is a $\oplus$-diagonal CAP space,
then every subspace is uniformly $\oplus$-strict.\end{prop}
\begin{proof}
Let $S\subseteq X$, let $\alpha\in[0,\infty]$ and take a filter on $ S^{(\alpha)}$ and call $\F$ is the filter generated on $X$. For each $x\in S^{(\alpha)}$, take $\S(x)\in\filter S$ such that
$\lambda(\S(x))(x)\leq\alpha$ and for $x \not \in S^{(\alpha)}$ let $\S(x) = \{x\}^{ \uparrow}.$ Since $X$ is $\oplus$-diagonal,
for $\G=\S(\F)$ we have
\[
\lambda(\G)=\lambda(\S(\F))\leq\lambda(\F)\oplus\bigvee_{x\in X}\lambda(\S(x))(x)\leq\lambda(\F)\oplus\alpha
\]
on $X$.
Clearly $S \in \bigcap_{x\in S^{(\alpha)}} \S(x)$ and since $S^{(\alpha)}$ belongs to $\F$ we have $S \in \G$. We finally check that $\G^{(\alpha)} \leq \F.$ Let $Z \in \bigcap _{x \in F \cap S^{(\alpha)}} \S(x)$ for some $F \in \F$. With $u \in F \cap S^{(\alpha)}$ the filter $\S(u)$ contains $Z$ and $\lambda \S(u)(u) \leq \alpha$. So $u \in Z^{(\alpha)}.$ It follows that $F \cap S^{(\alpha)} \subseteq Z^{(\alpha)}.$
\end{proof}
\begin{thm}
\label{thm:mainextensiondirect}Let $\alpha\in[0,\infty]$ and let
$Y$ be a CAP space. If $S$ is an $\alpha$-$\oplus$-strict subspace of a convergence
approach space $X$ and $f:S\rightarrow Y$ is a contraction, then
every $\oplus$-regular extension $g:h(S,f,\alpha)\to Y$ of $f$
satisfies $m_{\oplus}(g)\leq  \alpha\oplus\alpha$.\end{thm}
\begin{proof} 
We may assume $\alpha < \infty$. Let $g$ be an $\oplus$-regular extension  $g:h(S,f,\alpha)\to Y$.
Let $\F\in\filter\left(h(S,f,\alpha)\right)$ and $x_{0}\in h(S,f,\alpha)$.
Since $S$
is an $\alpha$-$\oplus$-strict subspace of $X$ there is a $\G\in\filter S$
such that $\G^{(\alpha)}\leq\F$ and 
$$\lambda_{X}(\G)(x_{0})\leq\lambda_{X}(\F)(x_{0})\oplus\alpha.$$
 Since $\G^{(\alpha)} \vee h(S,f,\alpha) \leq \F$ we have 
$$(f[\G])^{\alpha} \leq  g[\G^{(\alpha)} \vee h(S,f,\alpha)]\leq g[\F],$$
where the first inequality follows from the assertion $g(G^{\alpha} \cap h(S,f,\alpha)) \subseteq (f(G))^{\alpha}.$
Indeed for
$x\in G^{(\alpha)} \cap h(S,f,\alpha)$ 
there is an ultrafilter $\U$ on $G$ with
$\lambda(\U)(x)\leq\alpha$, that is with $\U\in H_{S}^{\alpha}(x)$.
Since $g$ is an admissible extension of $f$, $g(x)\in\bigcap\limits _{\beta\in[0,\infty]}F_{S}^{\beta}(x)$
so that in particular $g(x)\in F_{S}^{\alpha}(x)$ and $\lambda_{Y}(f[\U])(g(x))\leq\alpha$.
Thus $g(x)\in(f(G))^{(\alpha)}$.\\
Therefore 
\[
\lambda_{Y}(g[\F])(g(x_{0}))\leq\lambda_{Y}(f[\G])^{(\alpha)}(g(x_{0}))\leq\lambda_{Y}((f[\G]))(g(x_{0})) 
\oplus \alpha,
\]
since $g(x_0)$ is a regularity point of $Y.$ 
With $\lambda_{X}(\G)(x_{0})=\gamma,$ using the fact that $g(x_0) \in F_{S}^{\gamma}(x_{0})$ we obtain 
\[
\lambda_{Y}(f[\G])(g(x_{0}))\leq \gamma =\lambda_{X}(\G)(x_{0}).
\]
Finally we obtain 
$$\lambda_{Y}(g[\F])(g(x_{0}))\leq \lambda \G(x_0) \oplus \alpha \leq \lambda \F(x_0) \oplus \alpha\oplus \alpha$$
\end{proof}
\begin{cor}
\label{cor:mainextensionregularcase}If $S$ is a $\oplus$-strict
subspace of a CAP space $X$ and $Y$ is a $\oplus$-regular CAP space,
then every admissible extension $g:h(S,f)\to Y$ of a contraction
$f:S\rightarrow Y$ is a contraction.
\end{cor}
In view of Proposition \ref{prop:diagtostrict}, Corollary \ref{cor:mainextensionregularcase}
generalizes the following recent result of J{\"a}ger.
\begin{cor}
\cite[Theorem 3.2]{Jaegerextension}\label{thm:jager} Let $X$ be
a diagonal CAP space, $Y$ be a regular CAP space and let $S$ be
a subspace of $X$. Then for every contraction $f:S\rightarrow Y$
there exists a contraction $g:h(S,f)\to Y$ with $g_{|S}=f$.
\end{cor}
The restriction of Theorem \ref{thm:mainextensiondirect} to \textbf{Conv
}is essentially (in fact, it is slightly more general than) the direct
part of \cite[Theorem 1.1]{FricKent92}:
\begin{cor}
If $S$ is a strict subspace of a convergence space $X$ and $Y$
is a convergence space, then every regular extension $g:h(S,f)\to Y$
of a continuous map $f:S\to Y$ is continuous. In particular, if $Y$
is regular, every admissible extension $g:h(S,f)\to Y$ of a continuous
map $f:S\to Y$ is continuous.
\end{cor}

As for the converse, we have:
\begin{thm}
\label{thm:mainextensionconverse}If $Y$ is not $\oplus$-regular,
then there is a \emph{$\oplus$-approach space} $X$, a 
(uniformly $\oplus$-strict)
subspace $S$, a contraction $f:S\rightarrow Y$, an $\alpha\in[0,\infty)$,
and an admissible extension $g:h(S,f,\alpha)\to Y$ that is not contractive.\end{thm}
\begin{proof}
Since $Y$ is not $\oplus$-regular, there exists $A\neq\emptyset$,
$l:A\rightarrow Y$, $\S:A\rightarrow\filter Y$, $\H\in\filter A$,
and $y_{0}\in Y$ such that 
\begin{equation}
\lambda_{Y}(l[\H])(y_{0})>\lambda_{Y}(\S(\H))(y_{0})\oplus\bigvee_{a\in A}\lambda_{Y}(\S(a))(l(a)).\label{eq:notreg}
\end{equation}
 Let $X:=(Y\times A)\cup A\cup\{x_{\infty}\}$, $S:=Y\times A$, and
$f:S\to Y$ be $f(y,a)=y$. Let 
\[
\N:=\bigcup\limits _{H\in\H}\bigcap\limits _{a\in H}(\S(a)\times\{a\}^{\uparrow})\wedge\{a\}^{\uparrow}.
\]
On $X$, we define the following CAP structure: 
\begin{eqnarray*}
\lambda_{X}(\G)((y,a)) & := & \begin{cases}
0 & \mbox{if }\G=\{(y,a)\}^{\uparrow}\\
\infty & \mbox{otherwise}
\end{cases}\\
\lambda_{X}(\G)(a) & := & \begin{cases}
0 & \mbox{if }\G=\{a\}^{\uparrow}\\
\lambda_{Y}(\S(a))(l(a)) & \mbox{if \ensuremath{\G\geq\left(\S(a)\times\{a\}^{\uparrow}\right)\wedge\{a\}^{\uparrow}}\ensuremath{}and \ensuremath{\G\neq\{a\}^{\uparrow}}}\\
\infty & \mbox{otherwise}
\end{cases}\\
\lambda_{X}(\G)(x_{\infty}) & := & \begin{cases}
0 & \mbox{if }\G=\{x_{\infty}\}^{\uparrow}\\
\lambda_{Y}(\S(\H))(y_{0}) & \mbox{if }\G\geq\N\\
\infty & \mbox{otherwise}.
\end{cases}
\end{eqnarray*}
Clearly $X$ is a pre-approach space. Moreover $X$ satisfies the $\oplus$-diagonal condition, the verification of which is left to the reader, and thus it is a $\oplus$-approach space.
Note that $f$ is a contraction on $S,$ which by Proposition \ref{prop:diagtostrict} is a uniformly $\oplus$-strict subspace. Put 
$$\alpha:=\lambda_{Y}(\S(\H))(y_{0})\vee\bigvee_{a\in A}\lambda_{Y}(\S(a))(l(a))$$ which is finite by (\ref{eq:notreg}).
We claim that $h(S,f,\alpha)=X.$ 

Indeed, first observe that $A\subseteq h(S,f,\alpha).$ Since $\S(a) \times \{a\}^{\uparrow}$ contains $S$ and has $\lambda_{X} (\S(a) \times \{a\}^{\uparrow}) (a) \leq \alpha$ whenever $a \in A,$ we already have $A \subseteq S^{(\alpha)}.$
Moreover  for $a \in A$ we have $l(a)\in\bigcap\limits _{\epsilon\in[0,\infty]}F_{S}^{\epsilon}(a).$
This follows from the fact that for every $\G\geq(\S(a)\times\{a\}^{\uparrow})\wedge\{a\}^{\uparrow}$ on $S$ we have $\G\geq\S(a)\times\{a\}^{\uparrow}$ and hence $f[\G] \geq \S(a).$
This implies that $\lambda_{Y}(f[\G])(l(a))\leq \lambda_{Y}(\S(a))(l(a)) = \lambda_{X} (\G) (a)$. \\
Secondly observe that $x_{\infty}\in h(S,f,\alpha).$ Since the filter $\mathcal{K }=\bigcup\limits _{H\in\H}\bigcap\limits _{a\in H}(\S(a)\times\{a\})^{\uparrow}$ contains $S$ and is finer than $\N$ we have $\lambda_{X} \mathcal{K} (x_{\infty}) = \lambda_{Y} \S (\H) (y_0) \leq \alpha,$ so
 $x_{\infty}\in S^{(\alpha)}.$ Next we check that $y_0 \in\bigcap\limits _{\epsilon\in[0,\infty]}F_{S}^{\epsilon}(x_{\infty}).$ Let $\epsilon$ be finite and $\G$ a filter on $S$ with $\lambda_X (\G) (x_{\infty})\leq \epsilon.$ Then $\N \leq \G$  and $\S(\H) \leq f[\G].$ It follows that 
 $$\lambda_{Y}(f[\G])(y_{0})\leq\lambda_{Y}(\S(\H))(y_{0})=\lambda_{X}(\G)(x_{\infty})\leq\epsilon$$ and therefore $y_{0}\in F_{S}^{\epsilon}(x_{\infty})$.

From the previous calculations it follows that the extension $g:h(S,f,\alpha)\rightarrow Y$
of $f$ defined by $g_{|S}=f$, $g_{|A}=l$ and $g(x_{\infty})=y_{0}$ is admissible. We finally show that $g$ is not contractive.
Since the filter on $X$ generated by $\H$ is finer than $\N$ we have:
\begin{eqnarray*}
\lambda_{X}(\H)(x_{\infty}) = \lambda_{Y}(\S(\H))(y_{0}) &\leq& \lambda_{Y}(\S(\H))(y_{0}) \oplus \bigvee_{a\in A}\lambda_{Y}(\S(a))(l(a))\\ 
&<& \lambda_{Y}(l[\H])(y_{0}) = \lambda_{Y}(g[\H])(y_{0})
\end{eqnarray*}


\end{proof}
Note that if in the proof above $Y$ is a convergence space (considered
as a CAP space), then we can assume $\lambda_{Y}(\S(a))(l(a))$ to
be $0$ for each $a\in A$, and $\lambda_{Y}(\S(\H))(y_{0})$ to be
$0$ as well. Thus, $X$ is then a\emph{ topological} CAP space. Therefore,
we recover:
\begin{cor}
\cite[Theorem 1.1]{FricKent92} \label{cor:FK}A convergence space
$Y$ is regular if and only if, whenever $S$ is a strict subspace
of a convergence (equivalently, topological) space $X$ and $f:S\rightarrow Y$
is a continuous map there exists a continuous map $\bar{f}:h(S,f)\rightarrow(Y,\tau)$
with $\bar{f}_{|S}=f$.
\end{cor}
Since every subspace of a diagonal convergence space is strict, we
also recover:
\begin{cor}
\cite{cook68} \label{cor:Cook} A Hausdorff convergence space $Y$
is regular if and only if for every diagonal convergence space $X$,
every subspace $S$ of $X$, and every continuous map $f:S\rightarrow Y$
there exists a (unique) continuous map $\bar{f}:h(S,f)\rightarrow Y$
with $\bar{f}_{|S}=f$.
\end{cor}

\thanks{We want to thank Gavin Seal (Ecole Polytechnique de Lausanne) for
his useful comments. 

\end{document}